\documentclass[reqno,11pt]{amsart}[draft]
\usepackage{amscd,amsfonts,amsmath,amssymb,amstext,amsthm}

\usepackage[utf8]{inputenc}
\usepackage[T2A]{fontenc}
\usepackage[russian,english,]{babel}
\usepackage{cite}
\usepackage[unicode]{hyperref}
\usepackage{color}
\usepackage{xcolor}
\usepackage{comment}
\usepackage[normalem]{ulem}

\makeatletter
\def\@settitle{\begin{center}%
    \baselineskip14\p@\relax
    \bfseries
    \MakeUppercase{\@title}
  \end{center}%
}
\makeatother

\newtheorem{theorem}{Theorem}
\newtheorem{lemma}{\it Lemma}[section]
\newtheorem{proposition}{Proposition}[section]
\newtheorem{corollary}{\it Corollary}[section]
\theoremstyle{remark}

\newtheorem{remark}{Remark}[section]
\newtheorem{example}{\it Example}[section]
\theoremstyle{definition}

\newtheorem{definition}{\it Definiton}[section]
\numberwithin{equation}{section}
\DeclareMathOperator*\uplim{\overline{lim}}
\def\C{{\mathbb C}}

\def\Cn*{{\C^n}^*}
\def\CN*{{\C^N}^*}

\def\R{{\mathbb R}}

\def\vol{{\rm vol}}
\def\E{\:{\rm e}}
\def\re{{\rm Re}}
\def\im{{\rm Im}}

\begin{document}
\title{
On common zeros of entire functions of exponential growth
}
\author{B. Kazarnovskii}
\address {
Institute for Information Transmission Problems of the Russian Academy of Sciences
\newline
{\it kazbori@gmail.com}.}
\begin{abstract}
For systems of equations with an infinite set of roots, one can sometimes obtain
Kushnirenko-Bernstein-Khovanskii type theorem
if replace the number of roots by their asymptotic density.
We consider systems of entire functions with exponential growth in the space $\mathbb C^n$,
and calculate the asymptotic distribution of their common zeros
in terms of the geometry of convex sets in the space $\mathbb C^n$.
\end{abstract}
\maketitle
\tableofcontents
\section{Introduction}\label{introduction}
\subsection{Asymptotics of Zero Distribution}
Consider the set $X$ of common zeros of entire functions of exponential growth $f_1,\ldots,f_m$ in $\C^n$.
Let $X_{\rm reg}$ denote the set of points $z\in X$ such that the co-dimension of the analytic set $X$ in the neighborhood of the point $z$ is equal to $m$.
In the generic situation, $X_{\rm reg}=X$.
Further, we consider $X$ as a current $\mathfrak M(f_1,\ldots,f_m)$, i.e., as the integration of compactly supported differential forms of degree $2n-2m$ over the analytic set $X_{\rm reg}$.
For example, if $m=n$, then $\mathfrak M(f_1,\ldots,f_n)=\sum_x\delta(x)$,
where $\delta(x)$ is the delta function with support at an isolated root $x$ of the system $f_1=\ldots=f_n=0$.
For a fixed $t>0$, let us denote the function $f_j(tz)$ by $f_{t,j}(z)$.
If the limit of the currents $\frac{1}{t^m}\mathfrak M(f_{t,1},\ldots,f_{t,m})$ exists as $t\to+\infty$,
we call it the \emph{asymptotic distribution} of the variety of common zeros of entire functions of exponential growth $f_1,\ldots,f_m$.
\begin{example}\label{exSum1}
Let $f$ be a quasi-polynomial in one variable $z$ with spectrum $K$, i.e., a finite sum of the form
$$
f(z)=\sum_{\xi\in K\subset\C}p_\xi(z)\E^{\xi z},
$$
where $p_\xi$ are polynomials in the variable $z$.
Then the asymptotic distribution of zeros of the function $f$ exists.
In this case, the current of the asymptotic distribution is a distribution in $\C$.
If $K\subset\re\C$, we denote by $[\alpha,\beta]$ the minimal interval in $\re\C$ containing all points of the finite set $K$.
If the polynomials $p_\alpha, p_\beta$ are non-zero,
then the value of the asymptotic density current $\mathfrak M(f)$ on a compactly supported function $\varphi$ is equal to $\frac{\beta-\alpha}{2\pi}\int_{\im\C}\varphi\: dy$.
\end{example}
\subsection{Functions of Exponential Growth}
Recall that a linear functional on the space of holomorphic functions on the complex manifold $M$ is called an analytic functional on $M$; see \cite{Lel}.
Let ${\C^n}^*$ be the space of linear functionals in $\C^n$,
and $\mu$ be an analytic functional in ${\C^n}^*$,
i.e., $\mu$ is a linear functional on the space of entire holomorphic functions in ${\C^n}^*$.
For a fixed $z\in\C^n$, consider $\E^{\xi(z)}$ as an entire function with $\xi\in{\C^n}^*$ as its argument.
Define $\hat\mu(z)=\mu(\E^{\xi(z)})$.
It is worth noting that the entire function $\hat\mu$ on $\C^n$ is called the Borel-Fourier transform (or Laplace transform; see \cite{H1}) of the analytic functional $\mu$.
The Borel-Fourier transform is a bijective map
from the space of analytic functionals in ${\C^n}^*$ to the space of all entire functions of exponential growth in $\C^n$,
i.e., entire functions $g\colon\C^n\to\C$,
such that
$$
\exists (C,a)\colon\:\forall z\in\C^n\colon \vert g(z)\vert\leq C\E^{a\vert z\vert}.
$$
For an analytic functional $\mu$, there exists a compact set $K\subset M$ for which the following holds. For any open neighborhood $U$ of the compact set $K$, there exists a constant $C_U$ such that
\begin{equation}\label{eqSupp}
  \forall f\colon \vert\mu(f)\vert \leq C_U\sup_U\vert f(\xi)\vert.
\end{equation}
The compact set $K$ is called the support\footnote{There is no guarantee of the existence of the minimal support of an analytic functional; see \cite[Chapter 8]{Lel} or \cite[\S 4.5]{H1}.} of the functional $\mu$.

Let $\mathcal S(K)$ be the space of analytic functionals in ${\C^n}^*$ with support $K$.
Denote by $\hat{\mathcal S}(K)$ the space of entire functions in $\C^n$, consisting of Borel-Fourier transforms of functionals $\mu\in\mathcal S(K)$.
For $K\subset\re{\C^n}^*$, the exact description of functions in $\hat{\mathcal S}(K)$ is given by the Paley-Wiener theorem; see, for example, \cite[Theorem 1.7.7]{H2}.
\begin{example}\label{exQuasi1}
If the compact set $K\subset{\C^n}^*$ is finite, then the space $\hat{\mathcal S}(K)$ contains all quasi-polynomials with spectrum $K$, i.e., entire functions in $\C^n$ of the form
\begin{equation}\label{eqQuasi}
 \sum_{\xi\in K} p_\xi(z)\E^{\xi(z)},
\end{equation}
where the coefficients $p_\xi$ are polynomials in the variables $z_1,\ldots,z_n$.
\end{example}
Define the function $h_K\colon\C^n\to\R$ as $h_K(z)=\max_{\xi\in K}\re\xi(z)$.
Recall that $h_K$ is called the support function of the compact set $K$.
From (\ref{eqSupp}), it follows (see \S\ref{reg}) that
for any function $f\in\hat{\mathcal S}(K)$
\begin{equation}\label{eq<1}
 \forall z\in\C^n\colon\:\uplim_{t\to+\infty}\frac{\log\vert f(tz)\vert}{t}\leq h_K(z)
\end{equation}
\subsection{Averaged Distribution of Roots}\label{usr}
Let $V_1\subset\hat{\mathcal S}(K_1),\ldots,V_m\subset\hat{\mathcal S}(K_m)$ be finite-dimensional spaces.
Recall that we consider the variety of roots of the system
\begin{equation}\label{eqSys}
  f_1=\ldots=f_m=0; \,0\ne f_i\in V_i
\end{equation}
as the current
$\mathfrak M(f_1,\ldots,f_m)$.
Using arbitrary Hermitian inner products $\langle*,*\rangle_i$ in the spaces $V_i$,
we define the \emph{averaging
$\mathfrak M(V_1,\ldots,V_m)$
of the current} $\mathfrak M(f_1,\ldots,f_m)$
over all systems (\ref{eqSys});
see Definition \ref{dfAverage} in \S\ref{Aver}.
\begin{definition}\label{dfMeanAs}
For a finite-dimensional space $V\subset\hat{\mathcal S}(K)$, let $V(t)=\{f_t(z)=f(tz)\colon f\in V\}$.
For a Hermitian inner product $\langle*,*\rangle$ in $V$, define $\langle f_t,g_t\rangle=\langle f,g\rangle$.
If, in the topological space of currents in $\C^n$, the limit
$$
\mathfrak M^\infty(V_1,\ldots,V_m)=\lim_{t\to+\infty}\frac{1}{t^m}\mathfrak M(V_1(t),\ldots,V_m(t))
$$
exists and does not depend on the choice of inner products $\langle*,*\rangle_j$ in spaces $V_j$,
then we call it the \emph{averaged asymptotic distribution} of roots of systems (\ref{eqSys}).
\end{definition}
\subsection{Main Result}\label{mainResult}

Recall that, by definition, for a function $f$ on the complex manifold $M$, the value of the 1-form $d^cf$ on the tangent vector $\xi$ is equal to $df(-i\xi)$. If continuous functions $h_1,\ldots,h_m\colon M\to\R$ are plurisubharmonic functions on the manifold $M$, then $dd^ch_1\wedge\ldots\wedge dd^ch_m$ is a well-defined positive current of degree $2m$; see, for example, \cite{Lel}.

\begin{remark}
Currents of the form $dd^ch_1\wedge\ldots\wedge dd^ch_n$ are called Monge-Ampère measures. Such measures traditionally appear in problems related to the distribution of zeros of holomorphic systems of equations; see \cite{Lel,K0,Pass}.
\end{remark}

The support functions $h_i$ of the compacts $K_i\subset{\C^n}^*$ are convex and, consequently, plurisubharmonic. In particular, the current $dd^ch_1\wedge\ldots\wedge dd^ch_n$ is a non-negative measure in $\C^n$.

Below is defined the concept of a regular finite-dimensional subspace in $\hat{\mathcal S}(K)$; see Definition \ref{dfReg} in \S\ref{reg}, as well as an example at the end of \S\ref{mainResult}.

\begin{theorem}\label{thmMain}
 Let $V_1\subset\hat{\mathcal S}(K_1),\ldots,V_m\subset\hat{\mathcal S}(K_m)$ be regular spaces.
 Then the current of the averaged asymptotic distribution of roots of systems (\ref{eqSys}) exists and is given by the formula
 $$
\mathfrak M^\infty(V_1,\ldots,V_m)=\frac{1}{\pi^m}dd^ch_1\wedge\ldots\wedge dd^ch_m,
$$
where $h_i$ is the support function of compact $K_i$.
\end{theorem}
\begin{definition}\label{dfpseudo}
Let $h$ be the support function of a convex compact $K\subset{\C^n}^*$.
Put ${\rm pvol}(K)=\int_B(dd^ch)^n$, where $B$ is the ball of radius $1$ centered at the origin.
The value ${\rm pvol}(K)$ is called the \emph{pseudo-volume of the convex body} $K$.
If $h_1,\ldots,h_n$ are the support functions of convex compacts $A_1,\ldots,A_n$, then
$$
{\rm pvol}(A_1,\ldots,A_n)=\int_B dd^ch_1\wedge\ldots\wedge dd^ch_n
$$
is called the \emph{mixed pseudo-volume of convex bodies} $A_1,\ldots,A_n$.
\end{definition}
\begin{remark}
The pseudo-volume is a unitarily invariant valuation on convex bodies in complex vector spaces; see \cite{Alesk}.
\end{remark}

From Theorem \ref{thmMain}, the following statements follow.
\begin{corollary}\label{cor1}
Let $V_1,\ldots,V_n$ be regular subspaces in the spaces $\hat{\mathcal S}(K_1),\ldots,\hat{\mathcal S}(K_n)$, respectively. Then the asymptotic average, over all systems {\rm(\ref{eqSys})}, of the number of roots in a growing-radius ball centered at the origin is asymptotically given by
$$
{\rm pvol}\left({\rm conv}(K_1),\ldots,{\rm conv}(K_n)\right)\:r^n,
$$
where ${\rm conv}(K)$ denotes the convex hull of the compact $K$.
\end{corollary}
If $K\subset\re{\C^n}^*$, then the pseudo-volume of $K$ is equal to its volume as a body of full dimension in the $n$-dimensional real space $\re{\C^n}^*$.
Therefore we get the following statement; see \cite{K14f}.
\begin{corollary}
Let $K_1,\ldots,K_n\subset\re{\C^n}^*$. Then
$$
\mathfrak M^\infty(V_1,\ldots,V_n)=\frac{n!}{(2\pi)^n}\vol\left({\rm conv}(K_1),\ldots,{\rm conv}(K_n)\right)\:\mu_n,
$$
where $\vol(*,\ldots,*)$ is the mixed volume of convex bodies in $n$-dimensional real space $\re{\C^n}^*$, and $\mu_n$ is a measure in $\C^n$ with support in $\im\C^n$, defined as an integral of a function with respect to the Lebesgue measure on the space $\im\C^n$.
\end{corollary}
Finally, let us provide an example, which is the source of the concept of a regular space.
Later in \S\ref{mainResult}, it is assumed that the compact $K$ is a finite set.
Let $\hat {\mathcal S}_N(K)$ denote the space of quasi-polynomials of the form (\ref{eqQuasi}),
such that ${\rm deg}(p_\xi)\leq N$ for all $\xi\in K$. For example, $\hat {\mathcal S}_0(K)$ consists of exponential sums with spectrum $K$,
i.e., functions of the form $\sum_{\xi\in K, c_\xi\in\C\:} c_\xi\E^{\xi(z)}$.
\begin{theorem}\label{thmQuasi}
Any finite-dimensional space $V\subset\hat{\mathcal S}(K)$, such that $V\supset\hat {\mathcal S}_0(K)$, is regular.
In particular, for any $N\geq0$, the space of quasi-polynomials $\hat {\mathcal S}_N(K)$ is regular.
\end{theorem}
From here follows the old result on the asymptotic density of the variety of roots of systems of exponential sums \cite{K0,K1}.
\section{Averaged Distribution on a Complex Manifold}\label{Aver}
Let $V_1, \ldots, V_m$ be finite-dimensional spaces of holomorphic functions on an $n$-dimensional manifold $X$, and $\langle*,*\rangle_i$ be a Hermitian scalar product in the space $V_i$. The roots of the system
\begin{equation}\label{eqSys2}
  f_1=\ldots=f_m=0; \,0\ne f_i\in V_i
\end{equation}
depend solely on the projections $p_i=\pi_i(f_i)$ of the points $f_i\in V_i\setminus0$ onto the projectivizations $\mathbb P_i$ of the spaces $V_i$.
We consider the variety of roots of the system (\ref{eqSys2}) as a current $\mathfrak M(f_1,\ldots,f_m)$.
Furthermore, if necessary, we identify the system (\ref{eqSys2}) with a point $(p_1,\ldots,p_m)\in\mathbb P_1\times\ldots\times\mathbb P_m$.
Using the Hermitian metrics $\langle*,*\rangle_i$,
we define and calculate the averaging $\mathfrak M(V_1,\ldots,V_m)$ of the current $\mathfrak M(f_1,\ldots,f_m)$ over all systems (\ref{eqSys2}).
For this purpose, we employ the corresponding products of $\langle*,*\rangle_i$ and the Fubini-Study metrics in the spaces $\mathbb P_i$; see, for example, \cite{Shab}.
Let $\Omega_i$ denote the corresponding volume form in $\mathbb P_i$, normalized as $\int_{\mathbb P_i}\Omega_i=1$.
\begin{definition}\label{dfAverage} Suppose $\pi_i(f_i)=p_i$, where $\pi_i\colon V_i\setminus0\to\mathbb P_i$ is the projection mapping.
Identify the system (\ref{eqSys2}) with a point $(p_1,\ldots,p_m)\in \mathbb P_1\times\ldots\times\mathbb P_m$.
The current on $X$, defined as
$$
\mathfrak M(V_1,\ldots,V_m)=\int_{\mathbb P_1\times\ldots\times\mathbb P_m}\mathfrak M(f_1,\ldots,f_m)\: \Omega_1\wedge\ldots\wedge \Omega_m
$$
is called the averaged distribution of roots of systems of the form (\ref{eqSys2}).
\end{definition}
\begin{definition}\label{dfTheta}
Furthermore, we assume that
$$
  \forall (x\in X, i\leq m)\:\exists f\in V_i\colon f(x)\ne0
$$
Let $V^*_i$ be the space of linear functionals on $V_i$.
Define the mapping $\Theta_i\colon X\to V^*_i\setminus0$ as $\Theta_i(x)\colon f\mapsto f(x)$.
Consider the Hermitian product $\langle*,*\rangle^*_i$ in the space $V^*_i$,
conjugate to the product $\langle*,*\rangle_i$.
For $x\in X$, let $\|x\|_i=\sqrt{\langle \Theta_i(x),\Theta_i(x)\rangle^*_i}$.
\end{definition}
\begin{theorem}\label{crofton1}
It holds that
$$
\mathfrak M(V_1,\ldots,V_m)=\frac{1}{(2\pi)^m}dd^c\log\|x\|^2_1\wedge\ldots\wedge dd^c\log\|x\|^2_m
$$
\end{theorem}
The assertion of the theorem is local.
That is, for its proof, one can replace $X$ with any small neighborhood of the point $x\in X$.
Therefore, by Sard's lemma, the theorem is reduced to the case where the mapping
$$
X\xrightarrow{\Theta_1\times\ldots\times\Theta_m}(V^*_1\setminus0)\times\ldots\times (V^*_m\setminus0)\xrightarrow{\pi^*_1\times\ldots\times\pi^*_m}\mathbb P^*_1\times\ldots\times \mathbb P^*_m,
$$
where $\mathbb P^*_i$ is the dual of $\mathbb P_i$ projective space, and $\pi^*_i\colon V^*_i\to\mathbb P^*_i$ is the projection mapping, is a closed embedding.

Let $\omega^*_i$ be a Kähler form in $\mathbb P^*_i$, such that its integral over the projective line is equal to one.
Then the form $\frac{1}{2\pi}dd^c\log(\|f\|^*_i)^2$ in $V^*_i\setminus0$ is the pullback of the Kähler form $\omega^*_i$ in $\mathbb P^*_i$ under the projection mapping.
Therefore, Theorem \ref{crofton1} can be reduced to the following statement.
\begin{theorem}\label{thmcrofton} 
Let $X$ be a closed complex manifold with a boundary in $\mathbb P^*_1\times\ldots\times\mathbb P^*_m$, $p_i\in\mathbb P_i$, and $H(p_i)$ be a projective hypersurface in $\mathbb P^*_i$ defined by the equation $p_i(*)=0$. Consider $H(p_1,\ldots,p_m)=H(p_1)\times\ldots\times H(p_m)$. Then, for any differential form $\varphi$ of degree $2n-2m$ on the manifold $\mathbb P^*_1\times\ldots\times\mathbb P^*_m$, it holds that
$$
\int_{(p_1,\ldots,p_m)\in\mathbb P_1\times\ldots\times\mathbb P_m}\left(\int_{X\cap H(p_1,\ldots,p_m)}\varphi\right)\Omega_1\wedge\ldots\wedge\Omega_m=
\frac{1}{(2\pi)^m}\int_X\varphi\wedge\omega^*_1\wedge\ldots\wedge\omega^*_m
$$
\end{theorem}
Theorem \ref{thmcrofton} can be identified as the Crofton formula for the product of projective spaces, as announced in \cite{K1}.
Here is a brief derivation of Theorem \ref{thmcrofton}.

As is customary in integral geometry (see, for example, \cite{Sh}), Theorem \ref{thmcrofton} can be reformulated in the language of double fibrations.
Let
$$
\Gamma(\mathbb P)=\{(q,p)\colon p\in\mathbb P,q\in\mathbb P^*, p\in H(q)\},
$$
$\gamma(\mathbb P)\colon (p,q)\mapsto q$, and $\delta(\mathbb P)\colon(p,q)\mapsto p$ be the two projection mappings from $\Gamma(\mathbb P)$ to $\mathbb P^*$ and $\Gamma(\mathbb P)$ to $\mathbb P$, respectively. Consider the double fibration
 \begin{equation}\label{doubleMany}
   \mathbb P^*_1\times\ldots\times\mathbb P^*_m\xleftarrow{\gamma}\Gamma(\mathbb P_1)\times\ldots\times\Gamma(\mathbb P_m)\xrightarrow{\delta}\mathbb P_1\times\ldots\times\mathbb P_m,
 \end{equation}
where $\gamma=\gamma(\mathbb P_1)\times\ldots\times\gamma(\mathbb P_m)$,
 $\delta=\delta(\mathbb P_1)\times\ldots\times\delta(\mathbb P_m)$. Let $\delta^*$ and $\gamma_*$ be the mappings of the inverse and direct images of differential forms, corresponding to the mappings $\delta$ and $\gamma$. Then, Theorem \ref{thmcrofton} transforms into the equality
\begin{equation}\label{eqdouble}
 \gamma_*\delta^*\:\Omega_1\wedge\ldots\wedge\Omega_m=\omega^*_1\wedge\ldots\wedge\omega^*_m.
\end{equation}
The fibration (\ref{doubleMany}) is a direct product of $m$ double fibrations
$$
\mathbb P^*_i\xleftarrow{\gamma(\mathbb P_i)}\Gamma(\mathbb P_i)\xrightarrow{\delta(\mathbb P_i)}\mathbb P_i.
$$
Therefore, (\ref{eqdouble}) reduces to the case $m=1$, i.e., to the following statement.
\begin{proposition}
Let $\Gamma=\{(p,q)\in\mathbb P^*\times\mathbb P\colon\:\langle p,q\rangle=0\}$,
$\gamma=\gamma(\mathbb P)$, and $\delta=\delta(\mathbb P)$. Then it holds that
$\gamma_*\delta^*\Omega=\omega^*$.
\end{proposition}
The last proposition is a well-known formula of the Crofton type; see \cite{Sh}. Its proof is based on the commutativity of the action of the unitary group on $\mathbb P$, $\mathbb P^*$, and $\Gamma$ with the mappings $\gamma$ and $\delta$.
\section{Regular Subspaces}\label{reg}
Let $V$ be a finite-dimensional subspace in $\hat{\mathcal S}(K)$ with a fixed
Hermitian scalar product $\langle*,*\rangle$.
Let $B\subset V$ denote the ball of radius $1$ centered at the origin.
Below, we will define the concept of regularity for the space $V$.
In this definition, the Hermitian product $\langle*,*\rangle$ is used.
However, as it can be easily observed,
the regularity property of $V$ does not depend on the choice of $\langle*,*\rangle$.

Consider
the function $\max_{f\in B} \vert f(tz)\vert$ in $\C^n$ with parameter $t>0$.
According to Definition \ref{dfMeanAs},
the quantity $\max_{f\in B} \vert f(tz)\vert$ is equal to the norm of the linear functional $f_t\mapsto f_t(z)$ in the space $V(t)$.
\begin{definition}\label{dfReg}
Call the space $V$ regular if,
as $t\to+\infty$, the parameter-dependent function
$$
z\mapsto\frac{\log\max_{f\in B} \vert f(tz)\vert}{t}
$$
locally uniformly converges to the support function $h_K$ of the compact set $K$.
\end{definition}
\begin{example}\label{exNon}
Let $L\subset K$.
Then, if ${\rm conv}(K)\ne{\rm conv}(L)$,
any subspace $\hat{\mathcal S}(K)$ consisting of functions belonging to $\hat{\mathcal S}(L)$
is not regular.
Conversely, if ${\rm conv}(K)={\rm conv}(L)$,
any regular subspace $\hat{\mathcal S}(L)$
is also a regular subspace of $\hat{\mathcal S}(K)$.
\end{example}
\begin{proposition}\label{prLocUp}
For any arbitrarily small $\varepsilon>0$, there exists a constant $C$ such that for any $t>0$,
$$
\frac{\log\max_{f\in B} \vert f(tz)\vert}{t}\leq \frac{C}{t}+ h_K(z)+\varepsilon\vert z\vert.
$$
\end{proposition}
Proposition \ref{prLocUp} reduces the question of regularity to a lower bound for the function $\frac{\log\max_{f\in B} \vert f(tz)\vert}{t}$ depending on the growing parameter $t$.
\begin{corollary}\label{corF}
Assume that for any nonzero $z\in\C^n$, there exists a neighborhood $U_z$ of the point $z$ and a function $F_z\in V$ such that

\textbf{\rm (*)} as $t$ grows, the function
$\frac{\log\vert F_z(tw)\vert}{t}$
on $U_z$
uniformly converges to $h_K(w)$.

Then the space $V$ is regular.
\end{corollary}
\begin{proof}
Let $K_z$ denote the open cone in $\C^n$, consisting of points of the form $\{\tau w\colon w\in U_z\}$.
Then condition (*) is satisfied by replacing the neighborhood $U_z$ with the cone $K_z$.
Consider a finite covering of the space $\C^n$ by cones of the form $K_z$
and apply condition (*) to each of these cones.
\end{proof}
\begin{corollary}\label{corF2}
Assume that the space $V_1$ is regular, and $V_1\subset V\subset\hat{\mathcal S}(K)$.
Then the space $V$ is also regular.
\end{corollary}
\begin{proof}
The functions $F_z$ from Corollary \ref{corF} can be found in the subspace $V_1$.
\end{proof}
Next in \S\ref{reg}, the proof of Proposition \ref{prLocUp} is provided.
To do this, we use the topology of uniform convergence on compacts in the space of entire functions
and the corresponding weak topology in the space of analytic functionals in ${\C^n}^*$.
\begin{lemma}\label{lmReg1}
Let $A$ be a compact set of analytic functionals with support $K$.
Then for any open neighborhood $U$ of the compact $K$, there exists a constant $C_U$ such that for all $\mu\in A$,
$$
  \forall f\colon \vert\mu(f)\vert \leq C_U\sup_U\vert f(\xi)\vert
$$
\end{lemma}
\begin{proof}
Considering the compactness of the set $A$, it follows from (\ref{eqSupp}).
\end{proof}
\begin{lemma}\label{lmReg2}
Let $A$ be a compact set of analytic functionals with support $K$.
Then for any arbitrarily small $\varepsilon>0$, there exists a constant $C$ such that
$$
\forall \mu\in A\colon\:\vert \hat\mu(z)\vert \leq C \E^{h_K(z)+\varepsilon \vert z\vert}
$$
\end{lemma}
\begin{proof}
Let $U=K+\varepsilon B_1$, where $B_1\subset{\C^n}^*$ is a ball of radius $1$ centered at $0$. Since
1) $\sup_{\xi\in U}\vert\E^{\xi(z)}\vert=\sup_{\xi\in U}\E^{\re \xi(z)}=\E^{h_U(z)}$,
2) $h_U(z)=h_K(z)+\varepsilon\vert z\vert$, and
3) by definition, $\hat\mu(z)=\mu(\E^{\xi(z)})$,
the required estimate follows from Lemma \ref{lmReg1}.
\end{proof}
\begin{lemma}\label{lmReg3}
Let $A$ be a compact set of analytic functionals with support $K$.
Then for any arbitrarily small $\varepsilon>0$, there exists a constant $C$ such that for any $t>0$,
$$
\forall \mu\in A\colon\:\frac{\log\vert \hat\mu(tz)\vert}{t}\leq \frac{C}{t}+ h_K(z)+\varepsilon\vert z\vert.
$$
\end{lemma}
\begin{proof}
By Lemma \ref{lmReg2},
$
\vert \hat\mu(tz)\vert \leq C \E^{h_K(tz)+\varepsilon \vert tz\vert}
$.
Taking logarithms and using the homogeneity of the functions $\vert z\vert$ and $h_K$, we obtain the desired statement.
\end{proof}
\emph{Proof of Proposition} \ref{prLocUp}.
Let $A$ be the set of analytic functionals in ${\C^n}^*$ such that their image under the Borel-Fourier transform belongs to the unit ball $B\subset V$.
The Borel-Fourier transform is continuous and bijective. Therefore, the set $A$ is compact. Hence, the desired statement follows from Lemma \ref{lmReg3}.
\section{Proof of Theorem \ref{thmMain}}\label{theorem1}
Let us recall that, according to Theorem \ref{crofton1}, for the flow of the averaged distribution of roots $\mathfrak M(V_1,\ldots,V_m)$ (see Definition \ref{dfAverage}), it holds that
$$
\mathfrak M(V_1,\ldots,V_m)=\frac{1}{(2\pi)^m}dd^c\log\|x\|^2_1\wedge\ldots\wedge dd^c\log\|x\|^2_m,
$$
where $\|x\|_i(z)$ is the norm of the linear functional $f\mapsto f(z)$ on the space $V_i$; see Definition \ref{dfTheta}.
Let $\|x\|_{t,i}(z)$ denote the norm of the linear functional $f_t\mapsto f_t(z)$ on the space $V_i(t)$; see Definition \ref{dfMeanAs}.
Then, according to the definition of the asymptotic density flow (see Definition \ref{dfMeanAs}),
$$
\mathfrak M^\infty(V_1,\ldots,V_m)=\lim\frac{1}{t^m}\frac{1}{(2\pi)^m}dd^c\log\|x\|_{t,1}^2\wedge\ldots\wedge dd^c\log\|x\|_{t,m}^2.
$$
The regularity property of spaces $V_i$ means that the function $\frac{1}{t}\log\|x\|_{t,i}^2(z)$ locally converges uniformly to $2\log h_{K_i}(z)$ as $t$ grows; see Definition \ref{dfReg}.
Now the desired statement follows from the continuity property of the complex Monge-Ampère operator
$$
(g_1,\ldots,g_m)\mapsto dd^cg_1\wedge\ldots\wedge dd^cg_m
$$
with respect to the topology of locally uniform convergence of continuous plurisubharmonic arguments $g_i$; see, for example, \cite{BT,K14f}.
\section{Proof of Theorem \ref{thmQuasi}}\label{Quasi}
According to Corollary \ref{corF2}, it is sufficient to prove Theorem \ref{thmQuasi} in the case when $V=\hat{\mathcal S}_0(K)$, i.e., when $V$ is the space of exponential sums.

Let $0\ne z\in\C^n$.
Corollary \ref{corF} reduces the proof of Theorem \ref{thmQuasi} to the construction of
an exponential sum $F_z\in V$ such that, as $t$ grows, the function
$\frac{\log\vert F_z(tw)\vert}{t}$
converges uniformly to $h_K(w)$ in some neighborhood of the point $z\in\C$.
The construction of the function $F_z$ is given below.
\begin{definition}\label{dfTrunc}
For $0\ne z\in\C^n$, consider a real linear functional on ${\C^n}^*$
$\varphi_z(\xi)=\re \xi(z)$.
Associate with the point $0\ne z\in\C^n$ the face $\Delta(z)$ of the polytope ${\rm conv}(K)$
consisting of points $\xi\in{\rm conv} (K)$
where the functional $\varphi_z$ on ${\rm conv} (K)$ attains its maximum.
Call $\Delta(z)$ the supporting face of the point $z$.
\end{definition}
\begin{lemma}
If the point $x\in\C$ is sufficiently close to $z$, then the supporting face $\Delta(x)$ is a face of the polytope $\Delta(z)$.
\end{lemma}
\begin{proof}
Follows from the definition of supporting faces.
\end{proof}
%
%
\begin{lemma}
Let $\Delta_1=\Delta(z),\Delta_1,\ldots,\Delta_N$ be the set of all faces of the polytope $\Delta(z)$.
Choose a function
$F_z(w)=\sum_{\xi\in\Delta(z)\cap K,\:c_\xi\ne0}c_\xi\E^{\xi(w)}$ so that $\forall i\colon\sum_{\xi\in\Delta_i\cap K}c_\xi\E^{\xi(z)}\ne0$.
Then, the condition \textbf{\rm (*)} from Corollary {\rm\ref{corF}} is satisfied.
\end{lemma}
\begin{proof}
Let $\Delta_i$ be the supporting face that is close to $z$ of the point $x\in\C^n$.
Then, by construction, $\vert F_z(tx)\vert= A(t) \E^{h_K(t x)}+B(t)$ for all $t>0$, where $A(t),B(t)$ are bounded for all $t>0$.
Hence, it follows that
$\frac{\log\vert F_z(tx)\vert}{t}$
locally
uniformly converges to $h_K(x)$,
i.e., the condition \textbf{\rm (*)} is satisfied.
\end{proof}
%
%
%
%
%

\end {document}